\documentclass[11pt]{amsart}
\usepackage{geometry}
\usepackage{palatino}
\usepackage[all]{xy}
\usepackage{amsthm,amssymb,amsthm,amsmath,graphicx,color}
\usepackage{mathrsfs,amscd}
\usepackage{multirow}
\newtheorem{theorem}{Theorem}[section]
\newtheorem{conjecture}[theorem]{Conjecture}

\newtheorem{algorithm}[theorem]{Algorithm}
\newtheorem{corollary}[theorem]{Corollary}
\newtheorem{lemma}[theorem]{Lemma}
\newtheorem{proposition}[theorem]{Proposition}
\newtheorem{example}[theorem]{Example}
\newtheorem{remark}[theorem]{Remark}

\newcommand{\bb}[1]{\mathbb{#1}}
\newcommand{\mc}[1]{\mathcal{#1}}
\newcommand{\mf}[1]{\mathfrak{#1}}

\begin{document}
\title[Quantization of Symplectic Spherical Nilpotent Orbits]{Regular Functions of Symplectic Spherical Nilpotent Orbits and their Quantizations}
\author{Kayue Daniel Wong}
\address[Wong]{Department of Mathematics, Hong Kong University of Science and
technology, Clear Water Bay, Kowloon, Hong Kong}
\email{makywong@ust.hk}
\begin{abstract}
We study the ring of regular functions of classical spherical orbits $R(\mc{O})$ for $G = Sp(2n,\bb{C})$. In particular, treating $G$ as a real Lie group with maximal compact subgroup $K$, we focus on a quantization model of $\mc{O}$ when $\mc{O}$ is the nilpotent orbit $(2^{2p}1^{2q})$. With this model, we verify a conjecture by McGovern and another conjecture by Achar and Sommers related to the character formula of such orbits. Assuming the results in \cite{B 2008}, we will also verify the Achar-Sommers conjecture for a larger class of nilpotent orbits.
\end{abstract}
\maketitle

\section{Introduction}
Let $G$ be a complex simple Lie group. An element $X$ in the Lie algebra $\mf{g}$ is called \textit{nilpotent} if $ad(X)^n = 0$ for some large $n \in \mathbb{N}$. Nilpotent elements of the same conjugacy class form a nilpotent orbit. Motivated by Kirillov's Orbit Method, one would like to `attach' a unitary representation on every nilpotent orbit $\mc{O}$. More precisely, consider $G$ as a real Lie group with maximal compact subgroup $K$, one would like to find a (hopefully unitarizable) $(\mf{g}_{\bb{C}}, K_{\bb{C}})$ module $X_{\mc{O}}^+$ such that
\begin{align}
X_{\mc{O}}^+|_{K_{\bb{C}}} \cong R(\mc{O})
\end{align}
\noindent as $K_{\bb{C}} \cong G$-modules. We call this a \textit{quantization} of $\mc{O}$ (Our definition is analogous to a quantization scheme for some real nilpotent orbits using \textit{admissible data} in \cite{Yang}).\\

Let $G = Sp(2n, \bb{C})$. Then every nilpotent orbit can be parametrized by a partition $\lambda_1 \geq \lambda_2 \geq \dots \geq \lambda_k$ of $2n$, subject to the condition that $|\{ i | \lambda_i = 2r+1 \}|$ is even for all $r \in \bb{N}$ (see \cite{CM}). In this manuscript, we would like to find a quantization model of nilpotent orbits $\mc{O} = (2^{2p}1^{2q})$. It turns out the theory of \textit{special unipotent representations} will do the job. Indeed, Barbasch gave such quantization model for a much bigger class of classical nilpotent orbits in a preprint \cite{B 2008}. We will not use any tools there, but all results appearing below verify the results there. In particular, we have\\

\noindent \textbf{Theorem A} (Corollary \ref{cor:rounipotent}(a)). \textit{
Let $G = Sp(2n,\bb{C})$ and $\mc{O} = (2^{2p}1^{2q})$ be a nilpotent orbit. Then there are two special unipotent representations $X_{\mc{O}}^+$, $X_{\mc{O}}^-$ attached to $\mc{O}$, then as $G \cong K_{\bb{C}}$ modules,
$$X_{\mc{O}}^+ \cong R(\mc{O}).$$
Furthermore, writing $\tilde{\mc{O}}$ as the universal cover of $\mc{O}$. Then
$$X_{\mc{O}}^+ \oplus X_{\mc{O}}^- \cong R(\tilde{\mc{O}}).$$}

It turns out that the quantization model in Theorem A has a few applications:\\

\noindent \textbf{(I)\ A Character Formula of $R(\mc{O})$.\ }  McGovern in \cite{MG1} gave a description of $R(\mc{O})$ as follows:
\begin{theorem}[\cite{MG1} Theorem 3.1, Corollary 3.2] \label{thm:mc}
Let $\mc{O}$ be a complex nilpotent orbit in $\mf{g}$, with a choice of Jacobson-Morozov triple $\{e,f,h\}$. Write $\mf{g}_i$ as the $i$-eigenspace of $ad(h)$ on $\mf{g}$, and $\mf{q} = \sum_{i \geq 0} \mf{g}_i$. Then as $G$-modules,
$$R(\mc{O}) \cong Ind_T^G(\prod_{\alpha \in \Delta^+, \mf{g}_{\alpha} \subset \mf{g}_0 + \mf{g}_1} (1 - e^{\alpha})),$$
where $\Delta^+$ is the set of all positive roots of $\mf{g}$.
\end{theorem}
From now on, we will write $Ind_T^G(\lambda)$ instead of $Ind_T^G(e^{\lambda})$. Motivated by the study of Dixmier algebras, McGovern further conjectured the following:
\begin{conjecture}[\cite{MG1} Conjecture 5.1] \label{conj:mc}
For each nilpotent orbit $\mc{O}$, the $G$-structure of $R(\mc{O})$ can be expressed by
$$R(\mc{O}) \cong \sum_{w \in W_{\mc{O}}} c_w Ind_T^G(\mu - w\cdot \mu)$$
for a fixed character $\mu$, $c_w \in \mathbb{Q}$ and $W_{\mc{O}}$ is a subset of $W_G$, the Weyl group of $G$.
\end{conjecture}
For instance, if $\mc{O}$ is the principal nilpotent orbit, then $W_{\mc{O}} = \{Id\}$, $\mu = 0$ and $c_{Id} = 1$. If $\mc{O}$ is the trivial orbit, then $W_{\mc{O}} = W_G$, $\mu = \rho := \frac{1}{2}\sum_{\alpha \in \Delta^+} \alpha$ and $c_w = (-1)^{l(w)}$ ($l(w)$ is the length of the Weyl group element $w$ in its reduced form). Indeed, Theorem A gives an affirmative answer to Conjecture \ref{conj:mc}:\\

\noindent \textbf{Theorem B} (Corollary \ref{cor:rounipotent}(b)). \textit{
Conjecture \ref{conj:mc} holds for $\mc{O} = (2^{2p}1^{2q})$ in $Sp(2n,\bb{C})$. Furthermore, the same conjecture holds for the simply-connected cover $\tilde{\mc{O}}$ of $\mc{O}$.}\\

\noindent \textbf{(II)\ The Lusztig-Vogan bijection.}  The second application of the quantization model is related to the Lusztig-Vogan conjecture, which we will describe below.\\

Let $G$ be a connected complex simple Lie group, and let $\Lambda^+(G) \subset \mf{t}^*$ be the collection of highest dominant weights of finite dimensional representations of $G$. In an attempt of checking unitarity of certain classes of representations (see Lecture 8 of \cite{V1998}), Vogan conjectured that there is a bijection between the sets
$$\mc{N}_{o,r} := \{ (e,\tau)| e\ \text{nilpotent element,}\ \tau \in \widehat{G^e} \}/ \sim \ \longleftrightarrow \lambda \in \Lambda^+(G),$$
where $\widehat{G^e}$ is the set of all algebraic (finite-dimensional) irreducible representations of $G^e$, the stabilizer group of $e$ in $G$. Since the set $\mc{N}_{o,r}$ is defined up to conjugacy, we will denote any element in $\mc{N}_{o,r}$ by $(\mc{O},\sigma)$ instead of $(e, \sigma)/\sim$.\\

We now describe the construction of the conjectured map. Given $(\mc{O},\sigma) \in \mc{N}_{o,r}$, consider the expression (which is unique by Theorem 8.2 of \cite{V1998})
\begin{align} \label{eq:lusztigvogan}
Ind_{G^e}^G(\sigma) = \sum_{\lambda \in \Lambda^+(G)} m_{\lambda}(\mc{O},\sigma) Ind_{T}^{G}(\lambda),
\end{align}
where all but finitely many $m_{\lambda}(\mc{O},\tau) \in \bb{Z}$ are zero. Then Vogan's conjectured bijection map is given by
$$\gamma: (\mc{O},\sigma) \mapsto \lambda_{max},$$
with $\lambda_{max}$ being the maximal element in $\Lambda^+(G)$ such that $m_{\lambda}(\mc{O},\sigma) \neq 0$.\\

A priori this map is not well-defined, and the core of the problem is to make sense out of this map. In the case of $G = SL(n,\bb{C})$, this map is well-defined and made explicit by Achar in his Ph.D. thesis \cite{Ac1}. We will study the conjectured bijection when $G = Sp(2n,\bb{C})$, $\mc{O} = (2^{2p}1^{2q})$ and $\sigma$ is a one-dimensional representation of $G^e$.\\

Note that in the case when $\sigma = \mathrm{triv}$, $Ind_{G^e}^G(\mathrm{triv}) = R(\mc{O})$, and Conjecture \ref{conj:mc} gives the expression of $R(\mc{O})$ in the form of Equation \eqref{eq:lusztigvogan}. It turns out that Vogan's map can be read off much more easily using Conjecture \ref{conj:mc} then just using Theorem \ref{thm:mc}. As an evidence, Chmutova-Ostrik \cite{CO} attempted to compute Vogan's map for $(\mc{O},\mathrm{triv})$, using Theorem \ref{thm:mc} and some extra tools. Yet from the tables of their computed results, many orbits, especially the spherical orbits, are left blank. Our work will fill out some blanks left out by Chmutova-Ostrik:\\

\noindent \textbf{Theorem C} (Proposition \ref{prop:voganmap1}, \ref{prop:voganmap2}). \textit{
Let $\mc{O} = (2^{2p}1^{4r+2})$ be a nilpotent orbit in $Sp(2n,\bb{C})$, then
$$\gamma(\mc{O},\mathrm{triv}) = (2p+4r+2, 2p+4r, \dots, 2p+2, 2p, (2p-2)^2, \dots, 4^2, 2^2, 0)$$
(here the superscript denotes the multiplicity of the coefficient appearing in the expression). Similarly, for $\mc{O} = (2^{2p}1^{4r})$,
$$\gamma(\mc{O},\mathrm{triv}) = (2p+4r, 2p+4r-2, \dots, 2p+2, (2p-1)^2, (2p-3)^2, \dots, 3^2, 1^2).$$}

\noindent \textbf{(III)\ A conjecture by Achar and Sommers.}
Furthermore, Achar and Sommers made a conjecture on the Vogan's map in \cite{AS} related to the Lusztig-Spaltenstein dual ${}^L\mc{O}$ of $\mc{O}$. We briefly recall the conjecture here:\\

Let $\mc{N}_o$ be the set of all nilpotent orbits in $\mf{g}$, ${}^L{\mc{N}_o}$ be the set of all nilpotent orbits in the Langlands dual ${}^L\mf{g}$. In \cite{S2}, Sommers constructed a surjective map
$$d: \mc{N}_{o,c}\twoheadrightarrow {}^L\mc{N}_o,$$
where $\mc{N}_{o,c} = \{(\mc{O},C) | \mc{O} \in \mc{N}_o, C \subset \overline{A}(\mc{O})\ \text{conjugacy class} \}$. Here is a few features of the Sommers' map:\\
$\bullet$\ If we fix $C$ to be the trivial conjugacy class, then the above map is the Spaltenstein dual map.\\
$\bullet$\ For every ${}^L\mc{O} \in {}^L\mc{N}_o$, there is a special orbit $\mc{O} \in \mc{N}_o$ and a canonical conjugacy class $C$ such that $d(\mc{O},C) = {}^L\mc{O}$.\\
$\bullet$\ More precisely, if ${}^L\mc{O} \in {}^L\mc{N}_o$ is special, then the canonical preimage of $d$ is $(\mc{O},1)$, where $\mc{O}$ is the Spaltenstein dual to ${}^L\mc{O}$.\\

Fix ${}^L\mc{O} \in {}^L{\mc{N}_o}$ with $\mf{sl}_2$-triple $\{{}^Le, {}^Lf, {}^Lh\}$, and $(\mc{O},C)$ is the above-mentioned canonical preimage of the surjection $d$. In the case of classical groups, $C \subset \overline{A}(\mc{O}) \cong (\bb{Z}/2\bb{Z})^r$ is a single element. \\

Let $s_i$ be the non-trivial element in the $i^{th}$-copy of $(\bb{Z}/2\bb{Z})^r$. Then $C = \Pi_{i \in I} s_i$ for some subset $I \subset \{1,2,\dots,r\}$. Define
$$H_C := \langle s_i | i \in I \rangle \leq \overline{A}(\mc{O}),$$
and consider the preimage of $H_C$ under the quotient map $\pi: A(\mc{O}) \to \overline{A}(\mc{O})$, i.e. $K_C := \pi^{-1}(H_C)$. Then $K_C$, as a subgroup of the $G$-equivariant fundamental group $A(\mc{O})$, corresponds to an orbit cover $\mc{O}_{C} \cong G/G_C$ of $\mc{O}$. Then the Conjecture of Achar and Sommers is given by:

\begin{conjecture}(\cite{AS} Conjecture 3.1) \label{conj:AS}
Writing
$$R(\mc{O}_C) \cong Ind_{G_C}^{G}(\mathrm{triv}) = \sum_{\lambda \in \Lambda^+} m_{\lambda} Ind_{T}^G(\lambda)$$
as in the form of Equation \eqref{eq:lusztigvogan}, then the maximal element in the expression is equal to ${}^Lh$.
\end{conjecture}

\noindent \textbf{Theorem D} (Theorem \ref{thm:asproof}).
\textit{Conjecture 3.1 of \cite{AS} holds for $\mc{O} = (2^{2p}1^{2q})$ in $Sp(2n,\bb{C})$.}\\

In fact, using the results in \cite{B 2008}, one can also show that the conjecture holds for all classical nilpotent orbits satisfying the hypothesis of Theorem 2.0.6 of \cite{B 2008} (Theorem \ref{thm:b2008}).\\

\noindent \textbf{Notations.} Let $G$ be a complex simple Lie group with Borel subgroup $B$ and maximal compact subgroup $K$. Pick a and Cartan subgroup $H$ so that it contains a maximal torus of $K$. Write their corresponding Lie algebras as $\mf{g}$, $\mf{b}$,$\mf{k}$ and $\mf{h}$ respectively.\\

Treating $G$ as a real Lie group with maximal compact subgroup $K$, we can identify $\mf{g}_{\bb{C}} = \mf{g} \times \mf{g}$ with Cartan subalgebra $\mf{h}_{\bb{C}} = \mf{h} \times \mf{h}$, compact torus $\mf{t} = \{(x,-x) \in \mf{h}_{\bb{C}} | x \in \mf{h} \}$ and split torus $\mf{a}= \{ (x,x) \in \mf{h}_{\bb{C}} | x \in \mf{h} \}$ with $\mf{h} = \mf{t} \oplus \mf{a}$. Given $(\lambda_1, \lambda_2) \in \mf{h}^*$ so that $(\lambda_1 - \lambda_2)$ is a weight of a finite dimensional holomorphic representation of $G$, then $e^{(\lambda_1, \lambda_2)}$ can be treated as a character of $H$, and the \textit{principal series representation} with character $(\lambda_1,\lambda_2)$ is the $(\mf{g}_{\bb{C}}, K_{\bb{C}})$-module
$$X \left(
\begin{array}{cccccccccc}
\lambda_1 \\
\lambda_2 \\
\end{array}
\right) = K-\text{finite part of }Ind_B^G(e^{(\lambda_1,\lambda_2)} \otimes 1).$$
Since $e^{(\lambda_1, \lambda_2)}|_T = e^{\lambda_1-\lambda_2}$, the $K_{\bb{C}}$-types of $X \left(
\begin{array}{cccccccccc}
\lambda_1 \\
\lambda_2 \\
\end{array}
\right)$ is equal to $Ind_T^G(\lambda_1-\lambda_2)$.

\begin{proposition}[\cite{BV 1985} Proposition 1.8] \label{prop:BVprop}
Consider $(\lambda_1, \lambda_2)$, $(\lambda_1', \lambda_2') \in \mf{h}^* \times \mf{h}^*$ such that $\lambda_1 - \lambda_2$, $\lambda_1' - \lambda_2'$ are weights of some finite dimensional holomorphic representations of $G$. The following are equivalent:\\
$\bullet$\ $(\lambda_1, \lambda_2) = (w\lambda_1',w\lambda_2')$ for some $w \in W$, the Weyl group in $G$.\\
$\bullet$\ $X \left(
\begin{array}{cccccccccc}
\lambda_1 \\
\lambda_2 \\
\end{array}
\right)$ and $X \left(
\begin{array}{cccccccccc}
\lambda_1' \\
\lambda_2' \\
\end{array}
\right)$ have the same composition factors with same multiplicities.\\
$\bullet$\ The Langlands subquotient of $X \left(
\begin{array}{cccccccccc}
\lambda_1 \\
\lambda_2 \\
\end{array}
\right)$, written as $L \left(
\begin{array}{cccccccccc}
\lambda_1 \\
\lambda_2 \\
\end{array}
\right)$, is the same as that of $X \left(
\begin{array}{cccccccccc}
\lambda_1' \\
\lambda_2' \\
\end{array}
\right)$.\\
Furthermore, every irreducible $(\mf{g}_{\bb{C}}, K_{\bb{C}})$-modules is equivalent to some $L \left(
\begin{array}{cccccccccc}
\lambda_1 \\
\lambda_2 \\
\end{array}
\right)$.
\end{proposition}
Here is a couple of applications of the above Proposition in our following work:\\

\noindent \textbf{(a)}\ In formulas concerning only with the Grothendieck group of $(\mf{g}_{\bb{C}},K_{\bb{C}})$-modules, for instance the global character formulas, or $K_{\bb{C}}$-type decompositions, we use the equivalence $X \left(
\begin{array}{cccccccccc}
w\lambda_1 \\
w\lambda_2 \\
\end{array}
\right) \cong X \left(
\begin{array}{cccccccccc}
\lambda_1 \\
\lambda_2 \\
\end{array}
\right)$ in the Grothendieck group to obtain expressions of a more desirable form (e.g. Equations \eqref{eq:Re}, \eqref{eq:Rs} below).\\

\noindent \textbf{(b)}\ In writing the Langlands subquotient, we tacitly pick a suitable $w$ to rearrange the coefficients, so that $L \left(
\begin{array}{cccccccccc}
w\lambda_1 \\
w\lambda_2 \\
\end{array}
\right) \cong L \left(
\begin{array}{cccccccccc}
\lambda_1 \\
\lambda_2 \\
\end{array}
\right)$ is of our desirable form. (e.g. Theorem \ref{thm:AB}, Proposition \ref{prop:unipotenttheta})\\

We will describe elements in $\mf{h}^*$ using coordinates, e.g. $\epsilon_i = (0,\dots,0,\overbrace{1}^{i^{th}-\mathrm{coordinate}},0,\dots,0)$. Fix a simple root system $\Pi^+ = \{\epsilon_1 - \epsilon_2, \epsilon_2 - \epsilon_3, \dots, \epsilon_{n-1} - \epsilon_n, 2\epsilon_n\} \subset \Delta^+ \subset \mf{h}^*$, then every irreducible highest weight module of $G$ is parametrized by $V_{(a_1, a_2, \dots, a_n)}$, where $a_1 \geq a_2 \geq \dots \geq a_n \geq 0$ are non-negative integers.\\

Note also that the Weyl group $W \cong S_n \ltimes (\bb{Z}/2\bb{Z})^n$ acts on $\mf{h}^*$. We write $\alpha_1 \sim \alpha_2$ if two elements $\alpha_1$, $\alpha_2 \in \mf{h}^*$ are conjugate to each other by an element in $W$.

\section{Dual Pair Correspondence}
Recall the dual pair correspondence obtained by Adams and Barbasch in \cite{AB}.
\begin{theorem}(\cite{AB} Theorem 2.8) \label{thm:AB}
Let $G' = O(2p, \bb{C})$ and $G = Sp(2n,\bb{C})$ with $n = 2p+q$. Writing the dual correspondence map as $\theta$, then as $(\mf{g},K)$-modules,
\begin{align*}
\theta(\mathrm{triv}) &= L \left(
\begin{array}{cccccccccc}
p+q, p+q-1, \dots, 2, 1;& p-1, p-2, \dots, 1, 0\\
p+q, p+q-1, \dots, 2, 1;& p-1, p-2, \dots, 1, 0  \\
\end{array}
\right),\\
\theta(\det) &= L \footnotesize{ \left(
\begin{array}{cccccccccccc}
 p+q, p+q-1, \dots, p+1; &p, &p-1, &\dots, &1, &0, &-1, &\dots, &-p+1\\
 p+q, p+q-1, \dots, p+1; &p-1, &p-2, &\dots, &0, &-1, &-2 &\dots, &-p \\
\end{array}
\right)},
\end{align*}
\normalsize{where $\mathrm{triv}, \det$ are the trivial and determinant representations of $G'$, treated as $(\mf{g}', K')$-modules.}
\end{theorem}
\begin{proof}
The explicit image of $\theta$ is given precisely in Theorem 2.8 of \cite{AB}. Notice that in their notations, $\mu = \lambda_1 - \lambda_2$ and $\nu = \lambda_1 + \lambda_2$. In the notation of \cite{AB},
$$ \mathrm{triv} = L(\mu, \nu;\pm) = L(0, 2\rho; +),\ \ \ \det = L(0, 2\rho; -),$$
where $\rho = \frac{1}{2}\sum_{\alpha \in \Delta^+} \alpha = (p-1, p-2, \dots, 1, 0)$ (so that both $\lambda_1 = \frac{\mu+\nu}{2}$ and $\lambda_2 = \frac{\mu - \nu}{2}$ are both conjugate to $\rho$, the infinitesimal character of the trivial representation in $SO(2p)$). Then the result follows from directly applying Theorem 2.8 and the translation from $(\mu, \nu)$ to $(\lambda_1, \lambda_2)$ described above.
\end{proof}

\begin{remark}\mbox{}\\
\noindent (1)\ Note that $\theta(triv)$ is a spherical representation, being the quotient of $U(\mathfrak g)$ by the maximal ideal $I$ of its infinitesimal character, while $\theta(det)$ is a nonspherical $U(\mathfrak g)$-bimodule with annihilator $I$. In other words, they are both \textbf{unipotent representations} in the sense of Definition 5.23 of \cite{BV 1985} (see Proposition \ref{prop:unipotenttheta} below).\\

\noindent (2)\ In fact, the original definition of dual pair correspondence between $(G',G)$ $=$ $(O(2p,\bb{C})$, $Sp(2n,\bb{C}))$ should be the correspondence between the genuine irreducible $(\mf{g}_{\bb{C}}',$ $\tilde{K}_{\bb{C}}')$ and $(\mf{g}_{\bb{C}}, \tilde{K}_{\bb{C}})$ modules, where $\tilde{K_{\bb{C}}'}, \tilde{K_{\bb{C}}}$ are double covers of $K_{\bb{C}}'$ and $K_{\bb{C}}$ respectively. However, in the complex group case, both double covers split over $\bb{Z}/2\bb{Z}$, and hence all genuine irreducible $(\mf{g}_{\bb{C}}', \tilde{K}_{\bb{C}}')$ and $(\mf{g}_{\bb{C}}, \tilde{K}_{\bb{C}})$ modules can be characterized by $(\mf{g}_{\bb{C}}', K_{\bb{C}}')$ and $(\mf{g}_{\bb{C}}, K_{\bb{C}})$ modules (\cite{AB} p.4). We will be using this characterization for the rest of the paper.
\end{remark}
We are interested in studying the $K_{\mathbb{C}}$-type decomposition of both $\theta(\mathrm{triv})$ and $\theta(\det)$. Here are the results:
\begin{proposition} \label{prop:LZ}
As $K_{\bb{C}} \cong G = Sp(2n,\bb{C})$-modules,
\begin{align*}
\theta(\mathrm{triv}) &\cong \bigoplus_{\substack{m_i \in \bb{N}\cup\{0\}\\ m_1 \geq m_2 \geq \dots \geq m_{2p}}} V_{(2m_1, 2n_2, \dots, 2m_{2p}, 0, \dots, 0)},\\
\theta(\det) &\cong \bigoplus_{\substack{m_i \in \bb{N}\cup\{0\}\\ m_1 \geq m_2 \geq \dots \geq m_{2p}}} V_{(2m_1+1, 2m_2+1, \dots, 2m_{2p}+1, 0, \dots, 0)}.
\end{align*}
\end{proposition}
\begin{proof}
The first equality is given precisely by Theorem 2.4 and Section 4 of \cite{LZ3}. For the second equality, Theorem 2.2 of \cite{LZ3} gives the possibilities of the $K_{\bb{C}}$-types appearing in $\theta(\det)$ by studying the space of $K$-harmonics $\mc{H}(K)$ (\cite{KV}). One can conclude that the $K$-types of $\theta(\det)$ must be of the form $V_{(2m_1+1, 2m_2+1, \dots, 2m_{2p}+1, 0, \dots, 0)}$ as in the Proposition, and such $K$-types appear in $\theta(\det)$ with multiplicity at most one.\\

It remains to show that all such $K$-types appear in $\theta(\det)$. Using the notations in \cite{LZ1}, \cite{LZ2}, we need to study the distribution in $\mc{X} := M_{2p \times 2n}(\bb{C})$ given by
$$\partial_{2p} \delta := \det\left( \begin{array}{ccc}
\frac{\partial}{\partial z_{1,1}} & \dots & \frac{\partial}{\partial z_{1,2p}} \\
 & \ddots &  \\
\frac{\partial}{\partial z_{2p,1}} & \dots & \frac{\partial}{\partial z_{2p,2p}} \end{array} \right)\delta,$$
where $\delta$ is the Dirac distribution at the origin of $\mc{X}$. Then for any $h \in O(2p,\bb{C})$,
$$h \cdot \partial_{2p} \delta = \det(h) \partial_{2p} \delta.$$
Using Theorem 2.2(b) in \cite{LZ2} (which relied on Theorem 3.4 of \cite{LZ1}), the inner product
$$\langle \partial_{2p} \delta, w_{2m_1 +1, \dots, 2m_{2p}+1, 0, \dots, 0} \rangle \neq 0,$$
for the highest weight vector $w_{2m_1 +1, \dots, 2m_{2p}+1, 0, \dots, 0}$ of the $U(2n)$-type $(2m_1 +1, \dots, 2m_{2p}+1, 0, \dots, 0)$. Upon restricting our attention to the subgroup $K = Sp(2n) \subset U(2n)$, it generates an irreducible representation $V_{(2m_1 +1, \dots, 2m_{2p}+1, 0, \dots, 0)}$ in $K$. Hence the $K$-type $V_{(2m_1 +1, \dots, 2m_{2p}+1, 0, \dots, 0)}$ does appear in $K \cdot \partial_{2p} \delta$ (with multiplicity one). From there, we can copy the statement of Theorem 2 in \cite{LZ3} and conclude that there is a $(\mf{g}_{\bb{C}}, K_{\bb{C}})$-module isomorphism
$$\theta(\det) \cong \overline{\langle G \cdot \partial_{2p}\delta \rangle},$$
the closure of the span of $G \cdot \partial_{2p} \delta$ in the Frechet topology of the space of distributions $\mc{S}^*(\mc{X})$. In particular, every $K$-type of the form $V_{(2m_1 +1, \dots, 2m_{2p}+1, 0, \dots, 0)}$ does appear in $\theta(\det)$. So the Proposition is proved.
\end{proof}
On the other hand, we want to know the decomposition of $R(\mc{O})$ as $G$-modules. In fact, nilpotent orbits of the form $\mc{O} = (2^{l}1^{2q})$ are called \textit{spherical orbits}, having the property that a Borel subgroup $B$ of $G$ has a dense orbit in $\mc{O}$. These orbits have rich geometric and combinatoric structures. The classification of such orbits are given by Panyushev in \cite{Pan}. In particular, when $l = 2p$ is even, then $\mc{O} = (2^{2p}1^{2q})$ is a \textit{special} spherical orbit. The representations attached to such $\mc{O}$ are called \textit{special unipotent representations} in the sense of \cite{BV 1985}. We will study such representations in detail in the next Section.\\

The structure of $R(\mc{O})$ along with its universal cover $R(\tilde{\mc{O}})$ were given in \cite{Cos}.
\begin{proposition}[\cite{Cos}, Proposition 4.7, 4.8]
Let $\mc{O} = (2^l1^{2q})$ be a spherical nilpotent orbit in $G = Sp(2n,\bb{C})$. Then
\begin{align*}
R(\mc{O}) &\cong \bigoplus_{\substack{m_i \in \bb{N}\cup\{0\}\\ m_1 \geq m_2 \geq \dots \geq m_l}} V_{(2m_1, 2m_2, \dots, 2m_l, 0, \dots, 0)},\\
R(\tilde{\mc{O}}) &\cong R(\mc{O}) \oplus \bigoplus_{\substack{m_i \in \bb{N}\cup\{0\}\\ m_1 \geq m_2 \geq \dots \geq m_l}} V_{(2m_1+1, 2m_2+1, \dots, 2m_l+1, 0, \dots, 0)}.
\end{align*}
\end{proposition}
Combining the results in Proposition \ref{prop:LZ}, we get
\begin{corollary} \label{cor:roequalstheta}
Let $G = Sp(2n,\bb{C})$ and $\mc{O} = (2^{2p}1^{2q})$. Then as $G$-modules,
$$R(\mc{O}) \cong \theta(\mathrm{triv}),\ \ \ \ \ R(\tilde{\mc{O}}) \cong \theta(\mathrm{triv}) \oplus \theta(\det).$$
\end{corollary}

\section{Unipotent Representations}
In this section, we will describe precisely how we can construct unipotent representations from a classical special nilpotent orbit $\mc{O}$ with its dual ${}^L\mc{O}$ being an even orbit.
\begin{algorithm} \label{alg:specialunipotent}
Let $\mc{O}$ be a classical special nilpotent orbit such that its Spaltenstein dual ${}^L\mc{O}$ being an even orbit. Then the special unipotent representations attached to $\mc{O}$ can be constructed as follows:\\

\noindent \textbf{Step (1)} Determine $\lambda_{\mc{O}}$: take the Spaltenstein dual ${}^L\mc{O}$ of $\mc{O}$, and consider the semisimple part of the Jacobson-Morozov triple ${}^Lh$ of ${}^L\mc{O}$. Then $\lambda_{\mc{O}} := \frac{1}{2}{}^L{h}$.\\

\noindent \textbf{Step (2)} In \cite{Lu1} or Proposition 5.28 of \cite{BV 1985}, Lusztig gives a bijection between $\bar{A}(\mc{O})$ and a certain left cell of irreducible $W$-representations denoted as $V^L(w_0w_{\mc{O}})$ in \cite{BV 1985}. Write the bijection as
$$x \in \bar{A}(\mc{O}) \stackrel{1:1}\longleftrightarrow \sigma_x \in \hat{W}.$$
(here we used the fact that if $G$ is classical, $A(\mc{O})$ and $\overline{A}(\mc{O})$ are copies of $\bb{Z}/2\bb{Z}$, hence every conjugacy class of these groups is a singleton)\\

\noindent \textbf{Step (3)} Define $R_x$ as in \cite{BV 1985} by the $\sigma_x$-isotypic projection
 $$R_x := \frac{1}{|W_{\lambda_{\mc{O}}}|} \sum_{w \in W} tr(\sigma_x(w))X\left(
\begin{array}{cccccccccccc}
 \lambda_{\mc{O}}\\
 w\lambda_{\mc{O}} \\
\end{array}
\right).$$
Using \cite[Proposition 6.6]{BV 1985}, if $\sigma_x$ can be obtained as a truncated induction (see \cite[Chapter 11]{Car}) from the sign representation of a Weyl subgroup $W' \leq W$, i.e. $\sigma_x = j_{W'}^{W}(\mathrm{sgn})$, then we can reduce this formula into
 $$R_x := \sum_{w' \in W'} (-1)^{l(w')}X\left(
\begin{array}{cccccccccccc}
 \lambda_{\mc{O}}\\
 w'\lambda_{\mc{O}} \\
\end{array}
\right).$$

\noindent \textbf{Step (4)} Consequently, every unipotent representation corresponding to $\mc{O}$ is parametrized by $\pi \in \overline{A}(\mc{O})^{\wedge}$, and has the character formula
$$X_{\pi} := \frac{1}{|\overline{A}(\mc{O})|} \sum_{x \in \overline{A}(\mc{O})} tr(\pi(x)) R_x.$$
\end{algorithm}

\begin{example}
{\rm We now study the special unipotent representations attached to the orbit $\mc{O} = (2^{2p}1^{2q})$:\\
\noindent -- The Lusztig-Spaltenstein dual is given by ${}^L\mc{O} = (2p+2q+1, 2p-1, 1)$. Hence $$\lambda_{\mc{O}} = \frac{1}{2}\ ^Lh = (p+q, p+q-1, \dots, p, (p-1)^2, \dots, 2^2, 1^2, 0).$$

\noindent -- In \cite{Lu2}, Lusztig defined an injection
 $$\gamma(\mc{O}) := \{\mc{O}' \subseteq \overline{\mc{O}} | \mc{O}' \nsubseteq \overline{\mc{O}}_{spec}\ \text{for any other special orbit } \mc{O}_{spec} \subsetneq \mc{O} \} \hookrightarrow \overline{A}(\mc{O}).$$
\noindent Consider the following composition of maps:
$$\mc{O}' \in \gamma(\mc{O}) \hookrightarrow x(\mc{O}') \in \overline{A}(\mc{O}) \stackrel{Step (2)}\mapsto \sigma_{x(\mc{O}')} \in \hat{W}.$$
Then $\sigma_{x(\mc{O}')} = sp(\mc{O}')$, the Springer representation of $\mc{O}'$.\\
For $\mc{O} = (2^{2p}1^{2q})$, $\gamma(\mc{O}) = \{ \mc{O}, \mc{O}' \}$, where $\mc{O}' = (2^{2p-1}1^{2q+2})$, and $A(\mc{O}) = \overline{A}(\mc{O}) = \bb{Z}/2\bb{Z} = \{e, s\}$ with $e$ being the identity element. So the injection above is indeed a bijection, and
$$\mc{O} \leftrightarrow e\ ;\ \ \mc{O}' \leftrightarrow s.$$
According to the algorithm of computing Springer representations given in Section 7 of \cite{S1},
$$\sigma_e = (1^{p}, 1^{p+q} ) = j_{D_p \times C_{p+q}}^{C_n} (\mathrm{sgn}),\ \ \sigma_s = (1^{p+q+1}, 1^{p-1}) = j_{D_{p+q+1} \times C_{p-1}}^{C_n} (\mathrm{sgn}).$$

\noindent -- The two reduced formula is of the form
\begin{align} \label{eq:Re}
R_e = \sum_{w \in W(D_p \times C_{p+q})}(-1)^{l(w)} X\left(
\begin{array}{cccccccccc}
&p-1, \dots, 1, 0; & p+q, \dots, 2,1 \\
w(&p-1, \dots, 1, 0; & p+q, \dots, 2,1)\end{array}
\right),
\end{align}
\begin{align} \label{eq:Rs}
R_s = \sum_{w' \in W(D_{p+q+1} \times C_{p-1})}(-1)^{l(w')} X\left(
\begin{array}{cccccccccc}
&p+q, \dots, 1, 0; & p-1, \dots, 2,1 \\
w'(&p+q, \dots, 1, 0; & p-1, \dots, 2,1)\end{array}
\right).
\end{align}
(Here we have used Remark (a) after Proposition \ref{prop:BVprop}.)\\

\noindent -- The two special unipotent representations are of the form
$$X_{\mc{O}}^+ = \frac{1}{2}(R_e + R_s),\ \ X_{\mc{O}}^- = \frac{1}{2}(R_e - R_s).$$
}\end{example}

We now link the special unipotent representations obtained in the previous example with $\theta(\mathrm{triv})$, $\theta(\det)$.
\begin{proposition} \label{prop:unipotenttheta}
As $(\mf{g}_{\bb{C}}, K_{\bb{C}})$-modules,
$$X_{\mc{O}}^+ \cong \theta(\mathrm{triv}),\ \ X_{\mc{O}}^- \cong \theta(\det)$$
\end{proposition}
\begin{proof}
The proposition can be proved directly by tracing along the lines of Corollary 5.24 in \cite{BV 1985} and Theorem \ref{thm:AB}. We present another proof here. Note that the infinitesimal character of $X_{\mc{O}}^+$ and $X_{\mc{O}}^-$ are both ($W \times W$-conjugacy class of) $(\lambda_{\mc{O}}, \lambda_{\mc{O}})$. By Proposition \ref{prop:BVprop}, they must be of the form $L \left(
\begin{array}{cccccccccc}
\lambda_{\mc{O}} \\
w\lambda_{\mc{O}} \\
\end{array}
\right)$ for some $w \in W$. For $X_{\mc{O}}^+$,
$$X_{\mc{O}}^+ \cong \frac{1}{2}(R_e + R_s) = \frac{1}{2}[(X\left(
\begin{array}{cccccccccc}
\lambda_{\mc{O}} \\
\lambda_{\mc{O}}\end{array}
\right) + \dots ) + (X\left(
\begin{array}{cccccccccc}
\lambda_{\mc{O}} \\
\lambda_{\mc{O}}\end{array}
\right) + \dots )] = X\left(
\begin{array}{cccccccccc}
\lambda_{\mc{O}} \\
\lambda_{\mc{O}}\end{array}
\right) + \dots,$$
where the remaining terms are of the form $X\left(
\begin{array}{cccccccccc}
\lambda_{\mc{O}} \\
w\lambda_{\mc{O}}\end{array}
\right)$, $l(w) > 0$. Hence its lowest $K_{\bb{C}}$-type is $\lambda_{\mc{O}} - \lambda_{\mc{O}} = (0^n)$. On the other hand, the only $L \left(
\begin{array}{cccccccccc}
\lambda_{\mc{O}} \\
w\lambda_{\mc{O}} \\
\end{array}
\right)$ having lowest $K_{\bb{C}}$-type $(0^n)$ is $L \left(
\begin{array}{cccccccccc}
\lambda_{\mc{O}} \\
\lambda_{\mc{O}} \\
\end{array}
\right)$. By (Remark (b) after) Proposition \ref{prop:BVprop} and Theorem \ref{thm:AB},
$$X_{\mc{O}}^+ \cong L \left(
\begin{array}{cccccccccc}
\lambda_{\mc{O}} \\
\lambda_{\mc{O}} \\
\end{array}
\right) \cong L \left(
\begin{array}{cccccccccc}
p+q, \dots, 2, 1;& p-1, \dots, 1, 0\\
p+q, \dots, 2, 1;& p-1, \dots, 1, 0 \\
\end{array}
\right) \cong \theta(\mathrm{triv}).$$

\noindent For $X_{\mc{O}}^-$, a direct computation shows that the term
$$X\left(
\begin{array}{cccccccccc}
\lambda_{\mc{O}} \\
w_p\lambda_{\mc{O}}\end{array}
\right) := X\left(
\begin{array}{cccccccccc}
p+q, \dots, p+1;& p, p-1 & p-1, p-2 &\dots, &2,1, &1,0  \\
p+q, \dots, p+1;& p-1, p & p-2, p-1 &\dots, &1,2, &0,1 \end{array}
\right)$$
appears in the expression $X_{\mc{O}}^- \cong \frac{1}{2}(R_e - R_s)$ with coefficient 1. Hence $X_{\mc{O}}^-$ must have lowest $K_{\bb{C}}$-type smaller than or equal to $(1^{2p}0^q) \sim (\lambda_{\mc{O}} - w_p\lambda_{\mc{O}})$, that is, the lowest $K_{\bb{C}}$-type must be of the form $(1^{2i}0^{n-2i})$ with $i \leq p$. However, another direct computation shows that all terms of the form
$$\{ X\left(
\begin{array}{cccccccccc}
\lambda_{\mc{O}} \\
w\lambda_{\mc{O}}\end{array}
\right) |w \in W,\ (\lambda_{\mc{O}} - w\lambda_{\mc{O}}) \sim (1^{2i}0^{n-2i}),\ i < p \}$$
do not appear in the expression of $X_{\mc{O}}^-$. Therefore, the lowest $K_{\bb{C}}$-type of $X_{\mc{O}}^-$ must be $(1^{2p}0^q)$.\\

Suppose $X_{\mc{O}}^- \cong L\left(
\begin{array}{cccccccccc}
\lambda_{\mc{O}} \\
w\lambda_{\mc{O}}\end{array}
\right)$ for some $w \in W$, then we have
$$X\left(
\begin{array}{cccccccccc}
\lambda_{\mc{O}} \\
w\lambda_{\mc{O}}\end{array}
\right) = X_{\mc{O}}^- \oplus\ Y = X\left(
\begin{array}{cccccccccc}
\lambda_{\mc{O}} \\
w_p\lambda_{\mc{O}}\end{array}
\right) \oplus (X_{\mc{O}}^- - X\left(
\begin{array}{cccccccccc}
\lambda_{\mc{O}} \\
w_p\lambda_{\mc{O}}\end{array}
\right)) \oplus Y,$$
where $(X_{\mc{O}}^- - X\left(
\begin{array}{cccccccccc}
\lambda_{\mc{O}} \\
w_p\lambda_{\mc{O}}\end{array}
\right))$ and $Y$ are both elements in the Grothendieck group with lowest $K_{\bb{C}}$-type strictly bigger than $(1^{2p}0^q)$. By comparing the first and last expressions of the above equation, one can take $w = w_p$ and hence
$$X_{\mc{O}}^- \cong L\left(
\begin{array}{cccccccccc}
\lambda_{\mc{O}} \\
w_p\lambda_{\mc{O}}\end{array}
\right) = L\left(
\begin{array}{cccccccccc}
p+q, \dots, p+1;& p, p-1 & p-1, p-2 &\dots, &2,1, &1,0  \\
p+q, \dots, p+1;& p-1, p & p-2, p-1 &\dots, &1,2, &0,1 \end{array}
\right).$$
By (Remark (b) after) Proposition \ref{prop:BVprop} and Theorem \ref{thm:AB}, the result follows.
\end{proof}

Combining Proposition \ref{prop:unipotenttheta} and Corollary \ref{cor:roequalstheta}, we have
\begin{corollary} \label{cor:rounipotent} \mbox{}\\
(a)\ As $K_{\bb{C}}$-modules,
\begin{align*}
R(\mc{O}) \cong X_{\mc{O}}^+,\ \ \ \ \ R(\tilde{\mc{O}}) \cong X_{\mc{O}}^+ \oplus X_{\mc{O}}^- = R_e.
\end{align*}
(b) The global character formula of $R(\mc{O})$ and $R(\tilde{\mc{O}})$ are given as:
\begin{align} \label{eq:ro}
R(\mc{O}) \cong \frac{1}{2} \sum_{w \in W(D_p \times C_{p+q})}&(-1)^{l(w)} X\left(
\begin{array}{cccccccccc}
&p-1, \dots, 1, 0; & p+q, \dots, 2,1 \\
w(&p-1, \dots, 1, 0; & p+q, \dots, 2,1)\end{array}
\right)+\\
\frac{1}{2} \sum_{w' \in W(D_{p+q+1} \times C_{p-1})} &(-1)^{l(w')} X\left(
\begin{array}{cccccccccc}
&p+q, \dots, 1, 0; & p-1, \dots, 2,1 \\
w'(&p+q, \dots, 1, 0; & p-1, \dots, 2,1)\end{array}
\right),  \nonumber
\end{align}
\begin{align} \label{eq:rtildeo}
R(\tilde{\mc{O}}) \cong \sum_{w \in W(D_p \times C_{p+q})}(-1)^{l(w)} X\left(
\begin{array}{cccccccccc}
&p-1, \dots, 1, 0; & p+q, \dots, 2,1 \\
w(&p-1, \dots, 1, 0; & p+q, \dots, 2,1)\end{array}
\right).\ \ \ \ \ \ \  \
\end{align}
By taking the $K_{\bb{C}}$-module isomorphism $X\left(
\begin{array}{cccccccccc}
\lambda_1 \\
\lambda_2 \end{array}
\right) \cong Ind_T^G(\lambda_1 - \lambda_2)$, we obtain Theorem B.
\end{corollary}

\section{Lusztig-Vogan Conjecture}
We now see how Equations \eqref{eq:ro} and \eqref{eq:rtildeo} help solve the Conjecture of Achar and Sommers in \cite{AS} for $\mc{O} = (2^{2p}1^{2q})$.\\

\begin{theorem} \label{thm:asproof}
Conjecture \ref{conj:AS} holds for ${}^L\mc{O}$ with $\mc{O} = (2^{2p}1^{2q})$.
\end{theorem}
\begin{proof}
First of all, we already know from Step (1) of Algorithm \ref{alg:specialunipotent} that ${}^Lh = 2\lambda_{\mc{O}}$. On the other hand, since ${}^L\mc{O}$ is special, the canonical preimage of Sommers' map $d$ is just $(\mc{O}, 1)$, where $1 \in \overline{A}(\mc{O})$ is the trivial element. \\

Also, the map $\pi: A(\mc{O}) \to \overline{A}(\mc{O})$ is identity, since both spaces are equal to $\bb{Z}/2\bb{Z}$. Therefore, $K_C$ is the trivial group in $A(\mc{O})$, and $G_C = (G^e)_0$, the connected component of $G^e$. Hence
$$\mc{O}_C \cong G/G_C = G/(G^e)_0 \cong \tilde{\mc{O}},$$
the universal cover of $\mc{O}$. \\

Now, from Equation \eqref{eq:rtildeo} above, $\lambda_{max}$ is obtained by taking the longest element $w_0$ in both $D_p$ and $C_{p+q}$, so that
\begin{align*}
R(\tilde{\mc{O}}) \cong X_{\mc{O}}^+ \oplus X_{\mc{O}}^- &\cong X\left(
\begin{array}{cccccccccc}
\lambda_{\mc{O}} \\
\lambda_{\mc{O}}\end{array}
\right) + \dots + (-1)^{l(w_0)}X\left(
\begin{array}{cccccccccc}
\lambda_{\mc{O}} \\
-\lambda_{\mc{O}} \end{array}
\right)\\
&\cong Ind_T^G(\lambda_{\mc{O}} - \lambda_{\mc{O}}) + \dots + (-1)^{l(w_0)}Ind_T^G(\lambda_{\mc{O}} - (- \lambda_{\mc{O}}))\\
&\cong Ind_T^G(0) + \dots + (-1)^{l(w_0)}Ind_T^G(2\lambda_{\mc{O}}).
\end{align*}
Consequently, the maximal term in the above expression is $2\lambda_{\mc{O}} = {}^Lh$.
\end{proof}

One may ask a more refined question than Conjecture \ref{conj:AS}: Does the maximal term $2\lambda_{\mc{O}}$ appear in the expression of $X_{\mc{O}}^+$ or  $X_{\mc{O}}^-$? The observation made by Achar and Sommers in Remark 3.2 of \cite{AS} is that, if $\mc{O} = (2^21^{4r})$ (i.e. when $p=1, q=2r$), then $2\lambda_{\mc{O}}$ will appear in $X_{\mc{O}}^-$ rather than $X_{\mc{O}}^+$. We will extend their observation by using Equation \eqref{eq:ro}:
\begin{proposition} \label{prop:voganmap1}
Let $\mc{O} = (2^{2p}1^{2q})$ in $Sp(2n,\bb{C})$. Then the maximal term $2\lambda_{\mc{O}}$ appears in $X_{\mc{O}}^+$ iff $q$ is odd. As a consequence, the image of the Vogan map
$$\gamma(\mc{O}, \mathrm{triv}) = 2\lambda_{\mc{O}}\  \Leftrightarrow\ q \equiv 1(\mathrm{mod}\ 2).$$
\end{proposition}
\begin{proof}
As in the proof of Theorem \ref{thm:asproof}, the maximal term $2\lambda_{\mc{O}}$ shows up in Equation \eqref{eq:ro} only when
$$w_0 \in D_p \times C_{p+q},\ \ \ \ \ w_0' \in D_{p+q+1} \times C_{p-1}$$
are both maximal length elements in their respective Weyl groups. Therefore, if $l(w_0)$ and $l(w_0')$ are of the same parity, then the maximal term $2\lambda_{\mc{O}}$ will show up in Equation \eqref{eq:ro}, and vice versa.\\

In Type $C_n$, the parity of the maximal length element is equal to the parity of $n$, while it is always even for $D_n$. Hence,
$$l(w_0) \equiv l(w_0') (\mathrm{mod}\ 2)\  \Leftrightarrow\  p+q \equiv p-1 (\mathrm{mod}\ 2)\ \Leftrightarrow q \equiv 1 (\mathrm{mod}\ 2).$$
And the Proposition follows.
\end{proof}

We can also ask what is the maximal element $\gamma(\mc{O},\mathrm{triv})$ if $q = 2r$ is even. In this case, $\mc{O} = (2^{2p}1^{4r})$ and
$$2\lambda_{\mc{O}} = (2p+4r, 2p+4r-2, \dots, 2p+2, 2p, (2p-2)^2, \dots, 4^2, 2^2, 0).$$

\begin{proposition} \label{prop:voganmap2}
Let $\mc{O} = (2^{2p}1^{4r})$ be a nilpotent orbit in $Sp(2n,\bb{C})$. Then
$$\gamma(\mc{O},\mathrm{triv}) = (2p+4r, 2p+4r-2, \dots, 2p+2, (2p-1)^2, (2p-3)^2, \dots, 3^2, 1^2).$$
Note that $\gamma(\mc{O},\mathrm{triv})$ is of smaller length than $2\lambda_{\mc{O}}$.
\end{proposition}

\begin{lemma} \label{lem:voganlemma}
Proposition \ref{prop:voganmap2} holds when $r = 0$, i.e.  $\mc{O} = (2^{2p})$.
\end{lemma}
\begin{proof}
The orbit $\mc{O} = (2^{2p})$ is induced from the trivial orbit in the Levi subgroup $GL(n,\bb{C})$. This orbit is called \textit{strongly Richardson} in \cite{CO}, and $\gamma(\mc{O},\mathrm{triv})$ is known and computed there. Comparing the results in \cite{CO} and our formula of $X_{\mc{O}}^+$:
$$X_{\mc{O}}^+ \cong  \frac{1}{2}\sum_{w \in W(D_p \times C_{p})}(-1)^{l(w)} Ind_T^G[
(p-1, \dots, 1, 0; p, \dots, 2,1) - w(p-1, \dots, 1, 0; p, \dots, 2,1)]$$
$$+ \frac{1}{2} \sum_{w' \in W(D_{p+1} \times C_{p-1})}(-1)^{l(w')} Ind_T^G[
(p, \dots, 1, 0;  p-1, \dots, 2,1) - w'(p, \dots, 1, 0;  p-1, \dots, 2,1)]
$$
is equal to
\footnotesize{
$$\sum_{w \in W(A_{2p-1})}(-1)^{l(w)} Ind_T^G[
(\frac{2p-1}{2}, \dots, \frac{1}{2}, \frac{-1}{2}, \dots, \frac{-(2p-1)}{2}) - w(\frac{2p-1}{2}, \dots, \frac{1}{2}, \frac{-1}{2}, \dots, \frac{-(2p-1)}{2})].$$
}
\normalsize{The maximal norm is obtained by taking $w = w_0 \in W(A_{2p-1})$ in the above formula, which gives
$$\gamma(\mc{O},\mathrm{triv}) = ((2p-1)^2, (2p-3)^2, \dots, 3^2, 1^2)$$
as in our Proposition.\\
(Indeed, a similar equality was first noticed by Barbasch and Vogan in Section 11 of \cite{BV 1985})}.
\end{proof}

\noindent \textit{Proof of Proposition \ref{prop:voganmap2}.} By Equation \eqref{eq:ro},
\begin{align*}
R(\mc{O}) \cong \frac{1}{2} \sum_{w \in W(D_p \times C_{p+2r})}& (-1)^{l(w)} X\left(
\begin{array}{cccccccccc}
&p-1, \dots, 1, 0; & p+2r, \dots, 2,1 \\
w(&p-1, \dots, 1, 0; & p+2r, \dots, 2,1)\end{array}
\right)+\\
\frac{1}{2} \sum_{w' \in W(D_{p+2r+1} \times C_{p-1})} & (-1)^{l(w')} X\left(
\begin{array}{cccccccccc}
&p+2r, \dots, 1, 0; & p-1, \dots, 2,1 \\
w'(&p+2r, \dots, 1, 0; & p-1, \dots, 2,1)\end{array}
\right).  \nonumber
\end{align*}
We want to make the first $2r$ coordinates as large as possible: Apply
$$x: (p+2r,\dots,p+1,p,\dots,1) \mapsto (-(p+2r),\dots,-(p+1),p,\dots,1)$$
in $C_{p+2r}$ for the first expression, and
$$x': (p+2r,\dots,p+1,p,\dots,1,0) \mapsto (-(p+2r),\dots,-(p+1),p,\dots,1,0)$$
in $D_{p+2r+1}$ for the second expression. It can be checked that $l(x)$ and $l(x')$ has the same parity (this relies on the fact that $q=2r$ is even). So the first $2r$ coordinates of $\gamma(\mc{O},\mathrm{triv})$ must be of the form
$$(p+2r,\dots,p+1) - (-(p+2r),\dots,-(p+1)) = (2p+4r, 2p+4r-2, \dots, 2p+2)$$
as in the Proposition.\\

For the last $4p$ coordinates, note that
$$W(D_p \times C_{p}) \leq W(D_p \times C_{p+2r}), \ \ \ \ \  W(D_{p+1} \times C_{p-1}) \leq W(D_{p+2r+1} \times C_{p-1}).$$
By Lemma \ref{lem:voganlemma}, there exists a $y \in W(D_p \times C_{p})$ and a $y' \in W(D_{p+1} \times C_{p-1})$ such that $l(y) \equiv l(y')\ (\mathrm{mod} 2)$, and the maximal norm in the Lemma is attained. Embed
$$y \in W(D_p \times C_{p}) \hookrightarrow W(D_p \times C_{p+2r}), \ \ \ \ \ y' \in W(D_{p+1} \times C_{p-1}) \hookrightarrow  W(D_{p+2r+1} \times C_{p-1}).$$
Then the last $4p$ coordinates of $\gamma(\mc{O},\mathrm{triv})$ are of the form
$$((2p-1)^2, (2p-3)^2, \dots, 3^2, 1^2)$$
as in the Proposition.\\

Combining the above calculations, if we take
$$z:=(1,x)\cdot y \in W(D_p \times C_{p+2r}), \ \ \ \ \ z'=(x',1)\cdot y' \in W(D_{p+2r+1} \times C_{p-1}),$$
then $l(z) \equiv l(z')\ (\mathrm{mod} 2)$, and the term
$$Ind_T^G(2p+4r, 2p+4r-2, \dots, 2p+2, (2p-1)^2, (2p-3)^2, \dots, 3^2, 1^2)$$
must have non-zero coefficient in Equation \eqref{eq:ro}, and this must be the maximum term we can obtain in the expression.
\qed

\section{Final Remarks}
As mentioned in the Introduction, the preprint \cite{B 2008} deals with a bigger class of nilpotent orbits. Assuming the results there are valid, we get:
\begin{theorem}\label{thm:b2008}
Let $\mc{O}$ be a classical, special nilpotent orbit with $A(\mc{O}) = \overline{A}(\mc{O})$ and ${}^L\mc{O}$ being even. Then Conjecture \ref{conj:AS} holds for its Spaltenstein dual ${}^L\mc{O}$.
\end{theorem}
\noindent \textit{Sketch of Proof.}
To prove Conjecture \ref{conj:AS}, we need to compute $R(\mc{O}_C)$. Under our hypothesis, the canonical preimage of ${}^L\mc{O}$ is $(\mc{O},1)$, and $R(\mc{O}_C) = R(\tilde{\mc{O}}_{univ})$ with $\tilde{\mc{O}}_{univ}$ being the universal cover of $\mc{O}$. Now, Theorem 2.0.6 of \cite{B 2008} says
\begin{align} \label{eq:sumuni}
R(\tilde{\mc{O}}_{univ}) \cong \bigoplus_{X_{\mc{O}}^i \in \mf{X}(\lambda_{\mc{O}})} X_{\mc{O}}^i,
\end{align}
where $\mf{X}(\lambda_{\mc{O}})$ is the collection of all unipotent representations having infinitesimal character $(\lambda_{\mc{O}}, \lambda_{\mc{O}})$ $\in$ $\mf{h}^* \times \mf{h}^*$. The value of $\lambda_{\mc{O}} \in \mf{h}^*$ is stated in Theorem 2.0.6, and is equal to $\frac{1}{2}{}^Lh$. Using Algorithm \ref{alg:specialunipotent}, and a similar argument as in Theorem \ref{thm:asproof}, one can show that the maximal term on the right hand side of Equation \eqref{eq:sumuni} is $Ind_T^G(\lambda_{\mc{O}} - (-\lambda_{\mc{O}})) = Ind_T^G({}^Lh)$, and the Theorem is proved.
\qed
\begin{example} {\rm
Let $\mc{O} = (4,4,3,3,2,2,1,1)$ in $G = Sp(20,\bb{C})$. Then $A(\mc{O}) = \overline{A}(\mc{O}) = \bb{Z}/2\bb{Z} \times \bb{Z}/2\bb{Z}$ and the Spaltenstein dual of $\mc{O}$ is ${}^L\mc{O} = (9,5,5,1,1)$. According to \cite{B 2008} and Step (1) of Algorithm \ref{alg:specialunipotent},
$$\frac{1}{2}{}^Lh = \lambda_{\mc{O}} = (4,3,2^3,1^3,0^2).$$
According to the above Theorem, $R(\mc{O}_C) = R(\tilde{\mc{O}}_{univ}) \cong \bigoplus_{X_{\mc{O}}^i \in \mf{X}(\lambda_{\mc{O}})} X_{\mc{O}}^i$
is the sum of all special unipotent representations attached to $\mc{O}$. By Step (4) of Algorithm \ref{alg:specialunipotent},
$$\bigoplus_{X_{\mc{O}}^i \in \mf{X}(\lambda_{\mc{O}})} X_{\mc{O}}^i \cong R_e,$$
where $R_e$ is given by the formula in Step (3).\\

Using Section 7 of \cite{S1}, $\sigma_e = j_{C_4 \times D_3 \times C_2 \times D_1}^{W}(\mathrm{sgn})$ and hence
\begin{eqnarray}
R(\tilde{\mc{O}}_{univ}) \cong R_e \cong \sum_{w \in C_4 \times D_3 \times C_2 \times D_1}(-1)^{l(w)} X\left(
\begin{array}{cccccccccc}
&4321; & 210; & 21; & 0 \\
w(&4321; & 210; & 21; & 0)\end{array}
\right).
\end{eqnarray}
Using the same arguments as in previous sections, we can take $w = w_0$ for the longest element, and obtain the maximal term $\lambda_{\mc{O}} - (-\lambda_{\mc{O}}) = {}^Lh$.}
\end{example}
Finally, along the lines of Proposition 4.3-4.4, we can compute which $X_{\mc{O}}^i$ contains the term $Ind_T^G(2\lambda_{\mc{O}})$. Or more precisely, we can compute the maximal element appearing in each of $X_{\mc{O}}^i$. We omit the details here.

\section*{acknowledgements} The author would like to thank his Ph.D. advisor Dan Barbasch for introducing unipotent representations to the author. He would also like to thank the referees for several helpful suggestions of the work.

\end{document}